\documentclass[a4paper]{article}

\usepackage[english]{babel} \usepackage[utf8]{inputenc} \usepackage{amsmath} \usepackage{graphicx} \usepackage{amssymb} \usepackage{amsthm} \usepackage{mathtools} \usepackage{tikz-cd} \usepackage{mathrsfs} \usepackage[colorinlistoftodos]{todonotes} \usepackage{enumitem} \usepackage{dsfont}

\mathtoolsset{showonlyrefs}
\title{Linear completeness of trajectories in Sobolev spaces and the symmetrised polydisk} \author{Lyonell Boulton and Connor Evans} \date{19th April 2026}

\newtheorem{thm}{Theorem}[section] \newtheorem{lem}[thm]{Lemma}     \newtheorem{cor}[thm]{Corollary}  \newtheorem{rmk}[thm]{Remark} 

     \newcommand{\C}{\mathbb{C}}

\begin{document}

\maketitle

\begin{abstract}
We establish a framework to determine the linear completeness of families of non-linear trajectories in Hilbert spaces, which relies on an infinite analytic block Toeplitz operator formulation. By means of this approach, we show the linear completeness in Sobolev spaces of two families of classical functions.  One is the moving family of dilated Weierstrass functions. The other is the family of eigenfunctions of the Gross-Pitaevskii equation with trapping potential in an infinite square well. Our results provide a new insight on the formulation of general methods to examine this intriguing concept, bridging classical non-linear analysis and linear approximation theory.
\end{abstract}

 \section{Introduction}

Characterisations for the dilations of a fixed function to be a Riesz basis of $\mathrm{H}^0=\mathrm{L}^2(0,1)$, in terms of a symbol represented by a Dirichlet series, were obtained about 30 years ago by Hedenmalm, Lindqvist and Seip. Indeed, by virtue of \cite[Theorem~5.2]{HLS}, we know that the family is a Riesz basis if and only if all the zeros and poles of the symbol lie on the left hand side plane, away from the imaginary axis. 

In order to answer a question about the linear completeness of non-linear trajectories posed by Fraenkel in \cite{Fraenkel} about 40 years ago, we recently considered the operator formulation of these basis properties for dilations of a family of moving functions. See \cite{boulton2023completeness}.  Succinctly, we showed that the underlying function profile can be changed, as the mode of vibration increases, while the family remains a Riesz basis of $\mathrm{H}^0$ far away from the linear setting where it originated. This is an interesting extension of the framework of \cite{HLS} which deserves further investigation. 

Determining that the zeros and poles of a symbol lie away from the critical line is usually very difficult. In \cite{BoultonLord} and \cite{boulton_2019}, we showed that a ``hands on'' formulation of the basis question via infinite block Toeplitz operators associated to multiplicative shifts, coupled with a simple Neumann series argument, enables verification of the Riesz basis property in specific cases.  Our first goal below is to show that this infinite block Toeplitz formulation is also effective for determining completeness properties of non-linear ``spread'' trajectories of functions in general Hilbert spaces and that this, in turn, can be used to show the linear completeness of large perturbations of these trajectories.

This general framework is presented in sections~\ref{sec1} and \ref{sec2}. Our first key contribution is Lemma~\ref{theorem1}, where we give necessary and sufficient conditions for the invertibility of a multiplicative shift Toeplitz operator with infinite diagonal periodic analytic symbol. Our second key contribution is Corollary~\ref{BetterthanCristensen}, which establishes a sufficient condition for the invertibility of its perturbations. By virtue of this corollary and a suitable representation via the symmetrised polydisc, found in Corollary~\ref{cor2}, in Section~\ref{sec2} we re-formulate a Paley-Wiener type perturbative criterion of Christensen found in \cite{Ch1995} (or \cite[Theorem~15.1.1]{Christensen}), adapting it to the context of trajectories of dilated functions.  The remainder of the paper is then devoted to the sharpening of the thresholds for completeness of two moving families of classical functions by the implementation of this re-formulation.

Section~\ref{sec3} addresses our second goal. We give a fairly complete characterisation of the basis question in the Sobolev spaces $\mathrm{H}^\alpha=\mathrm{H}^\alpha(0,1)$, where $\alpha\geq 0$,  for families of dilated Weierstrass functions. Our main conclusion is formulated in Theorem~\ref{theorem4.1}, where we show that the family always renders a Riesz basis in a regime up to a threshold. Moreover, it continues to do so, when subject to eventual periodicity of the underlying sequence. Notably, we show that a family comprising Weierstrass functions with different degrees of regularity, where each member has a certain period, forms a Riesz basis of $\mathrm{H}^\alpha$ for all $\alpha$ below the lowest degree of regularity among the members of this family. No previous investigation seems to exist in this direction. 

Our third goal, pursued in Section~\ref{sec4}, is to improve the currently known thresholds of linear completeness in $\mathrm{H}^\alpha$ for the eigenfunctions of the Gross-Pitaevskii equation, subject to Dirichlet boundary conditions on a segment. This equation models the stationary states of a Bose-Einstein condensate in a 1D square well trapping potential. In Theorem~\ref{oneterm} we find regions for the family to be a basis, in terms of the solution to a transcendental equation involving Lambert series and the regularity degree $\alpha\geq 0$. Then, in Theorem~\ref{twoterm}, we show that the regime is not optimal and improve it further, subject to periodicity.

Due its natural link with the theory of Hilbert space models, we include in Appendix~\ref{appa} a precise formulation of part of our framework in terms of the Schur-Cohn criterion of Pt\'{a}k and Young, \cite{ptakyoung}.


\section{Spread bases and the symmetrised polydisc\label{sec1}} We denote the open disk by $\mathbb{D}=\{z\in\mathbb{C}:|z|<1\}$ and the $d$-polydisc by $\mathbb{D}^d=\mathbb{D}\times\cdots\times \mathbb{D}$ with its standard Cartesian product norm. Let $\underline{\lambda}=(\lambda_{1},\hdots,\lambda_{d})\in\mathbb{D}^{d}$. Let $\pi_{d} :\mathbb{D}^{d}\rightarrow \mathbb{C}^{d}$ be given by $$\pi_{d}(\underline{\lambda})=\big(\pi_{d,1}(\underline{\lambda}),\hdots,\pi_{d,d}(\underline{\lambda})\big)\quad \text{where} \quad \pi_{d,k}(\underline{\lambda})=\sum_{1\leq j_{1}<\hdots<j_{k}\leq d} \lambda_{j_{1}}\cdots \lambda_{j_{k}}.$$ Below, the \emph{symmetrised polydisc} is the image $ \mathbb{G}_{d}=\pi_{d}\big(\mathbb{D}^{d}\big).$ Since $\pi_d$ is analytic in each component, then $\mathbb{G}_d$ is an open set and its closure $\operatorname{cl}(\mathbb{G}_d)=\pi_{d}\Big(\operatorname{cl}(\mathbb{D}^{d})\Big)$.  

We will write $\pi_{d,0}(\underline{\lambda})=1$ for all $d\in\mathbb{N}$ and use the shorthand $a_k,\,b_k$ or $c_k$ instead of $\pi_{d,k}(\underline{\lambda})$, suppressing the explicit reference to the coefficients $d$ or $\underline{\lambda}$, when the context allows it.  By construction, the scalar polynomial \[P(z)=\sum_{k=0}^{d}\pi_{d,d-k}(\underline{\lambda})z^{k}\] has the factorisation \begin{equation}\label{polynomialwithzerosinD1} P(z)=\prod_{j=1}^{d}(z+\lambda_{j}).   \end{equation} Moreover, the \emph{conjugate polynomial} \[Q(z)=z^d\overline{P(1/\overline{z})}=\sum_{k=0}^{d}\overline{\pi_{d,k}(\underline{\lambda})}z^{k}\] has factorisation $$Q(z)=\prod_{j=1}^{d} (1+\overline{\lambda_{j}}z).$$ Identifying $Q$ with $\alpha$, we therefore have that, a polynomial $\alpha(z)=1+\sum_{k=1}^da_kz^k$ has all its roots in $\mathbb{C}\setminus \operatorname{cl}(\mathbb{D})$ if and only if $\underline{a}\in\mathbb{G}_{d}$.  
 
In order to support the arguments given in section~\ref{sec3}, it might be convenient to visualise the subsets $\mathbb{G}_d\cap \mathbb{R}^d$ for $d=2,\,3$. For this purpose, see the figure~\ref{fig1}. We will provide a full characterisation of $\mathbb{G}_d$ in terms of the so-called Pt\'{a}k-Young matrices in appendix~\ref{appa}.

\begin{figure} \includegraphics[width=0.45\textwidth]{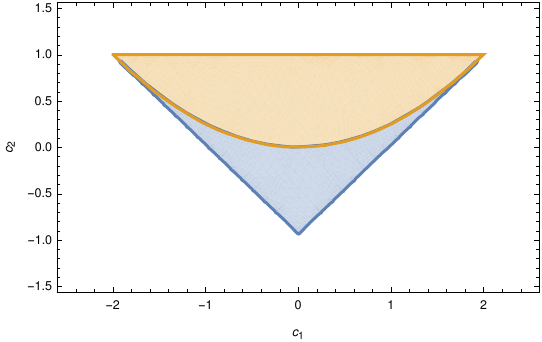} \hspace{1cm} \includegraphics[width=0.42\textwidth]{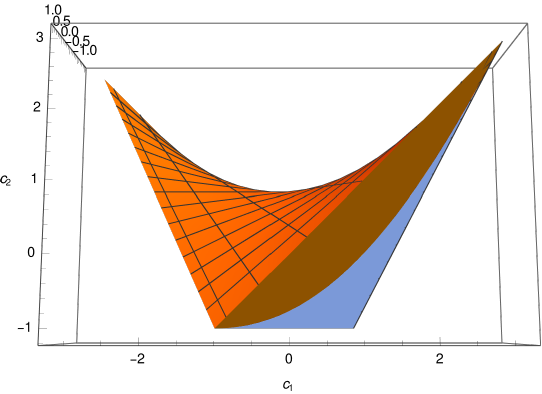} \caption{Subset $\mathbb{G}_d\cap \mathbb{R}^d$ projected onto $\mathbb{R}^d$. \underline{Left} $d=2$. For $(\lambda_1,\lambda_2)\in\mathbb{C}^2$ such that $\max\{|\lambda_1|,|\lambda_2|\}\leq 1$ we know $\pi_{2}(\underline{\lambda})=(\lambda_1+\lambda_2,\lambda_1\lambda_2)$. The condition $\pi_{2}(\underline{\lambda})\in\mathbb{R}^2$, implies that $\operatorname{Im}\lambda_1=-\operatorname{Im} \lambda_2$ and $\operatorname{Im}\lambda_1\operatorname{Im} \lambda_2=-\operatorname{Im}\lambda_2\operatorname{Im} \lambda_1$. Then, the projection has two components: two real roots (blue) comprising $\pi_2\big(\mathbb{D}^2\cap\mathbb{R}^2\big)$ and a conjugate pair roots (orange) comprising the $\underline{\lambda}\in\mathbb{D}^2$ with $\lambda_1=\overline{\lambda_2}=x+iy$ for reals $x^2+y^2\leq 1$ mapped by $h(x,y)=\big(2x,x^2+y^2\big)$. \underline{Right} $d=3$. In this case there are also two components forming the projection: (blue) three real roots and (orange) one real root and one conjugate pair. \label{fig1}} \end{figure}

\medskip

Let $\{e_n\}_{n=1}^\infty$ be an orthonormal basis of a separable Hilbert space $\mathcal{H}$ over the complex field. Let $M_j:\mathcal{H}\longrightarrow \mathcal{H}$ be such that $M_je_n=e_{nj}$. That is $M_j$ are \emph{spread} or \emph{multiplicative} shifts on $\mathcal{H}$ by a factor $j\in\mathbb{N}$. Below we consider the framework of Hardy classes with respect to these shift operators as formulated in \cite{RR}. 

We employ the following convention for diagonal operators on $\mathcal{H}$ with respect to $\{e_n\}_{n=1}^\infty$. For a scalar sequence $\{a(n)\}_{n=1}^\infty\subset \mathbb{C}$ such that \[\sup_{n\in\mathbb{N}}|a(n)|<\infty,\] the \emph{diagonal operator} $A:\mathcal{H}\longrightarrow \mathcal{H}$, denoted as \[A=\operatorname{diag}[a(n)]_{n=1}^\infty,\] is the bounded operator such that $Ae_n=a(n)e_n$. It is readily seen that \[\|A\|=\sup_{n\in\mathbb{N}}|a(n)|.\] Here and elsewhere below, whenever the context makes it unambiguous, we will write $\|B\|=\|B\|_{\mathcal{B}(\mathcal{H})}$ for a bounded linear operator $B$.

Let $\{A_k\}_{k=1}^\infty\subset \mathcal{B}(\mathcal{H})$ be a sequence of diagonal operators, where each \[A_k=\operatorname{diag}[a_k(n)]_{n=1}^\infty.\] Assume that there is an analytic scalar function $\alpha_n:\mathbb{D}\longrightarrow \mathbb{C}$ for each $n\in\mathbb{N}$, with \[\alpha_n(z)=1+\sum_{k=1}^\infty a_k(n)z^k \qquad \text{for }z\in\mathbb{D},\] such that \begin{equation} \label{eq3} \sup_{\substack{z\in\mathbb{D}\\n\in\mathbb{N}}} |\alpha_n(z)|<\infty.\end{equation} Then,  \begin{equation} \label{eq5}\mathbb{A}(z)=I+\sum_{k=1}^\infty A_k z^k=\mathrm{diag}[\alpha_n(z)]_{n=1}^\infty\end{equation} is an analytic family of bounded operators, $\mathbb{A}:\mathbb{D}\longrightarrow \mathcal{B}(\mathcal{H})$. Moreover, $\mathbb{A}\in H^{\infty}\big(\mathbb{D};\mathcal{B}(\mathcal{H})\big)$ with \[\|\mathbb{A}\|_{H^{\infty}(\mathbb{D};\mathcal{B}(\mathcal{H}))}=\sup_{z\in\mathbb{D}}\|\mathbb{A}(z)\|_{\mathcal{B}(\mathcal{H})}=\sup_{\substack{z\in\mathbb{D}\\n\in\mathbb{N}}} |\alpha_n(z)|.\]  Here and elsewhere, $H^{q}\big(\mathbb{D};\mathcal{B}(\mathcal{H})\big)=H^q_{\mathcal{B}(\mathcal{H})}(\mathbb{D})$ where the latter is the standard in \cite{RR}. 

\medskip

Let $p\in\mathbb{N}$ and let bounded $A=\operatorname{diag}[a(n)]_{n=1}^\infty$. It is readily seen that \begin{equation} \label{eq1} AM_p=M_pA \quad \iff \quad a(pn)=a(n)\text{ for all } n\in\mathbb{N}.\end{equation} Indeed, note that $AM_pe_n=a(pn)e_{pn}$ and that $M_pAe_n=a(n)e_{pn}$. This elementary observation will play a significant r\^{o}le below. Note that \eqref{eq1} holds true if and only if, it holds true simultaneously for all prime factors of $p$. 

Let \[\mathcal{L}_0=\mathrm{Ker}(M_p^*)=\mathrm{Span}\{e_j:j\not\in p\mathbb{N}\}\] be the \emph{wandering space associated to} $M_p$. If each of the operator coefficients $A_k$ defining the bounded analytic family $\mathbb{A}$ as above also satisfy \eqref{eq1}, then $\mathcal{L}_0$ is an invariant subspace for this family, $\mathbb{A}:\mathbb{D}\longrightarrow\mathcal{B}(\mathcal{L}_0)$, and $\mathbb{A}\in H^{\infty}\big(\mathbb{D};\mathcal{B}(\mathcal{L}_0)\big)$ with \[\|\mathbb{A}\|_{H^{\infty}(\mathbb{D};\mathcal{B}(\mathcal{L}_0))}=\sup_{\substack{z\in\mathbb{D}\\n\in\mathbb{N}}} |\alpha_n(z)|.\] Indeed, for the latter, note that $\|A\|_{\mathcal{B}(\mathcal{L}_0)}=\|A\|_{\mathcal{B}(\mathcal{H})}$ for any diagonal operator $A$ satisfying \eqref{eq1}.  

Below we write $\mathfrak{A}:H^2(\mathbb{D};\mathcal{L}_0)\longrightarrow H^2(\mathbb{D};\mathcal{L}_0)$ to denote the associated operator of multiplication, $(\mathfrak{A} \Phi)(z)=\mathbb{A}(z)\Phi(z)$, for analytic vectors $\Phi:\mathbb{D}\longrightarrow\mathcal{L}_0$ belonging to the Hardy class $H^2(\mathbb{D};\mathcal{L}_0)$. Recall  the fundamental identity \cite[section~1.15 (Theorem~B)]{RR},  \[\|\mathfrak{A}\|_{\mathcal{B}(H^2(\mathbb{D};\mathcal{L}_0))}=\|\mathbb{A}\|_{H^{\infty}(\mathbb{D};\mathcal{B}(\mathcal{L}_0))}.\]

Let $\mathcal{L}_k=M_{p^k}\mathcal{L}_0=M_p^k \mathcal{L}_0=M_p\mathcal{L}_{k-1}$, so that \[\mathcal{H}=\bigoplus_{k=0}^\infty \mathcal{L}_k.\] Let the infinite block representation of the linear operator \begin{equation} \label{eq2} T=I+\sum_{k=1}^\infty A_kM_{p^k},\end{equation} in this orthogonal decomposition, be denoted by \[T(\mathbb{A}):\bigoplus_{k=0}^\infty \mathcal{L}_k\longrightarrow \bigoplus_{k=0}^\infty \mathcal{L}_k.\] By virtue of \eqref{eq3} and \eqref{eq1}, $T$ is a Toeplitz operator relative to $M_p$ \cite[section~5.2 (Theorem~C)]{RR} and \[\|T\|_{\mathcal{B}(\mathcal{H})}=\|T(\mathbb{A})\|_{\mathcal{B}(\bigoplus_{k=0}^\infty\mathcal{L}_k)}  =\|\mathfrak{A}\|_{\mathcal{B}(H^2(\mathbb{D};\mathcal{L}_0))}. \] Therefore, \[\|T\|_{\mathcal{B}(\mathcal{H})} = \sup_{\substack{z\in\mathbb{D}\\n\in\mathbb{N}}} |\alpha_n(z)|\] and, as we shall see, the next invertibility criterion is valid.

\begin{lem}\label{theorem1} Let $p\in\mathbb{N}$. Let $\{A_k\}_{k=1}^\infty\in\mathcal{B}(\mathcal{H})$ be a sequence of diagonal operators whose entries satisfy \eqref{eq3} and \eqref{eq1}. Assume that the associated family of bounded operators \eqref{eq5} is analytic. Then, the operator $T$ in the representation \eqref{eq2} is invertible, if and only if \[s(T)=\inf_{\substack{z\in\mathbb{D}\\n\in\mathbb{N}}} |\alpha_n(z)| >0.\] Moreover, $\|T^{-1}\|=\frac{1}{s(T)}.$\end{lem}
\begin{proof} Let $T$ have the representation \eqref{eq2} and consider the infinite block Toeplitz operator framework given in the previous paragraphs. We show that $\mathbb{A}$ is invertible in the Banach algebra $H^{\infty}\big(\mathbb{D};\mathcal{B}(\mathcal{L}_0)\big)$ by exhibiting the analytic family of operators that give the inverse. The key point is that this algebra is commutative. In spite of this being elementary, we give most details of the construction for later purposes. 

Assume the hypothesis. Let $\mathbb{B}(z)=\sum_{k=0}^\infty B_kz^k$, where \begin{equation} \label{invexp} \begin{gathered}B_0=I,\qquad B_1=-A_1,\qquad B_2=A_1^2-A_2 \qquad \text{and} \\B_k=-(A_1B_{k-1}+A_2B_{k-2}+\cdots+A_k B_0).   \end{gathered} \end{equation} Then, \[\mathbb{A}(z)\mathbb{B}(z)=\mathbb{B}(z)\mathbb{A}(z)=I\] for all $z\in\mathbb{D}$. Moreover, $\mathbb{B}(z)$ is a family of diagonal operators and \[\mathbb{B}(z)=\mathrm{diag}[\beta_n(z)]_{n=1}^\infty,\] where \[\beta_n(z)=1+\sum_{k=1}^\infty b_k(n)z^k=\frac{1}{\alpha_n(z)}\] have coefficients $b_k(n)\in\mathbb{C}$ with expression in terms of $a_k(n)$ satisfying a recursion analogous to \eqref{invexp}. Note that $B_k=\mathrm{diag}[b_k(n)]_{n=1}^{\infty} $ satisfy \eqref{eq1}. 

Since \[0<s(T)\leq \inf_{z\in\mathbb{D}}|\alpha_n(z)|,\] for all fixed $n\in\mathbb{N}$, then each $\beta_n(z)$ is a bounded analytic function on $\mathbb{D}$. Moreover, \[\sup_{\substack{z\in\mathbb{D}\\n\in\mathbb{N}}}|\beta_n(z)|=\frac{1}{\inf_{\substack{z\in\mathbb{D}\\n\in\mathbb{N}}}|\alpha_n(z)|}=\frac{1}{s(T)}<\infty.\] Thus, $\mathbb{B}\in H^{\infty}\big(\mathbb{D};\mathcal{B}(\mathcal{L}_0)\big)$, $\mathbb{B}$ is the inverse of $\mathbb{A}$ in this Banach algebra and \[\|T^{-1}\|_{\mathcal{B}(\mathcal{H})}=\|\mathbb{B}\|_{H^{\infty}(\mathbb{D};\mathcal{B}(\mathcal{L}_0))}=\frac{1}{s(T)}.\]   
   \end{proof}

We are now ready to single-out the case of $\mathbb{A}$ being a polynomial and link with the framework of the symmetrised polydisc given at the beginning of this section.

\begin{cor} \label{cor2} Let $d\in\mathbb{N}$ and $p\in\mathbb{N}$. For $k=1,\ldots,d$ let $A_k=\operatorname{diag}[a_k(n)]_{n=1}^\infty$ be diagonal operators satisfying \eqref{eq1}. Let $T=I+\sum_{k=1}^d A_kM_{p^k}\in\mathcal{B}(\mathcal{H})$. Then, $T$ is invertible if and only if  there exists an open set $U\subset \mathbb{C}^d$ such that $\operatorname{cl}(U) \subset \mathbb{G}_d$ and \[(a_1(n),\cdots,a_d(n))\in U\] for all $n\in\mathbb{N}$.\end{cor}

\begin{proof} The condition \eqref{eq3} holds automatically, due to the fact that $A_k$ are diagonal and hence bounded operators. 

Now, the existence of the stated open subset $U$ is equivalent to the fact that the roots of the polynomials $\alpha_n(z)=1+\sum_{k=1}^d a_k(n)z^k$, all lie in the complement of an open set $V\subset \mathbb{C}$ such that $\operatorname{cl}(\mathbb{D})\subset V$. Let \[W=\{\lambda\in\mathbb{C}:\alpha_n(\lambda)=0 \text{ for some }n\in\mathbb{N}\}.\] Then, $W\subset \mathbb{C}\setminus V$ if and only if $s(T)>0$. Note that the latter is the hypothesis required in Lemma~\ref{theorem1}.   \end{proof}

Finally in this section, we formulate a stability criterion for the invertibility of combinations of infinite block Toeplitz operators of the above form. Although elementary, this result will be useful in the next sections to determine the range of parameters for a moving dilated family to form a Riesz basis. Note that here we do not assume that $B_j$ commutes with any of the $M_j$, hence there is no underlying algebra of Toeplitz operators for $T$. 

\begin{cor} \label{BetterthanCristensen}
Let $d\in\mathbb{N}$ and $p\in\mathbb{N}$. For $k=1,\ldots,d$ let $A_k=\operatorname{diag}[a_k(n)]_{n=1}^\infty$ be diagonal operators satisfying \eqref{eq1}. Let $\{B_j\}_{j=2}^\infty$ be another sequence of diagonal operators $B_j=\operatorname{diag}[b_j(n)]_{n=1}^\infty$. Assume that \begin{equation} \label{pertcond} \sum_{j=2}^\infty \sup_{n\in\mathbb{N}} |b_j(n)|<\inf_{\substack{z\in\mathbb{D}\\ n\in\mathbb{N}}} \left|1+\sum_{k=1}^d a_k(n)z^k\right|. \end{equation} Then, the operators $B_j$ are bounded and the operator \[T=I+\sum_{k=1}^d A_kM_p^k+\sum_{j=2}^\infty M_j B_j\] is bounded and invertible.
\end{cor}
\begin{proof}
Let \[S=I+\sum_{k=1}^d A_kM_p^k \qquad \text{and}\qquad R=\sum_{j=2}^\infty M_j B_j.\] According to Lemma~\ref{theorem1}, $S$ is invertible and \[\|S^{-1}\|^{-1}=s(S)=\inf_{\substack{z\in\mathbb{D}\\ n\in\mathbb{N}}} |\alpha_n(z)|.\] The right hand side of this is non-zero by the hypothesis. Also, since $\|B_j\|=\sup_{n\in\mathbb{N}}|b_j(n)|<\infty$, the operators $B_j$ are bounded and \[ \|R\|\leq \sum_{j=2}^\infty \|B_j\|<\|S^{-1}\|^{-1}.\] Hence, $\|R\|\|S^{-1}\|<1$ and, since $T=(I+RS^{-1})S$, indeed $T$ is a bounded invertible operator.  
\end{proof}


\section{Linear completeness of trajectories\label{sec2}} 
For $\alpha\geq 0$, denote by $\mathrm{H}^\alpha=\mathrm{H}^\alpha(0,1)$ the Sobolev space, where $f\in \mathrm{H}^\alpha$ iff \[\sum_{n=1}^\infty n^{2\alpha}|\hat{f}(n)|^2<\infty.\] Here and everywhere below, the Fourier coefficients of $f:(0,1)\longrightarrow \mathbb{R}$ are given by \[\hat{f}(n)=\sqrt{2}\int_0^1 f(x)\sin(n\pi x)\,\mathrm{d}x\] and we fix the norm as \[\|f\|_{\alpha}=\sqrt{\sum_{n=1}^\infty n^{2\alpha}|\hat{f}(n)|^2}.\] The associated inner product is \[ \langle f,g\rangle_{\alpha}=\sum_{n=1}^\infty n^{2\alpha}\hat{f}(n)\overline{\hat{g}(n)}.\] 

Our goal for the remaining of the paper will be to examine the following problem about a notion of linear completeness for families of trajectories on Hilbert spaces, studied by Fraenkel in \cite{Fraenkel}. \emph{Let $\mathrm{G}=\{\Gamma_n\}_{n=1}^\infty$ be an ordered family of weakly continuous curves, $\Gamma_n:(0,c)\longrightarrow \mathrm{H}^\alpha$. Find conditions on a subset $\mathtt{S}\subseteq(0,c)^\infty$ which ensure that for all $\{s_n\}_{n=1}^\infty\in \mathtt{S}$ the sequence of functions $\{g_n\}_{n=1}^\infty$ given by $g_n=\Gamma_n(s_n)$ is a basis of $\mathrm{H}^{\alpha}$.} That is, $\mathtt{S}$ characterises a generic choice of functions on each of the $\Gamma_n$, so that they form a basis. We will focus on trajectories $\Gamma_n$ formed by dilations of specific families of functions depending on a parameter, as we shall discuss next. Here the constant $c>0$ might depend on $\alpha$. 

Let \begin{equation}\label{onb} h_n(x)=\frac{\sqrt{2}}{n^\alpha}\sin(n\pi x).\end{equation} Then, $\{h_n\}_{n=1}^\infty\subset\mathrm{H}^\alpha$ is an orthonormal basis. For $s\in(0,c)$, let $s\longmapsto f_s\in \mathrm{H}^\alpha$ be a family of functions such that $\widehat{f_s}(1)=1$. Let $\Gamma_n(s)=\frac{1}{n^\alpha} f_s(n\,\cdot)$ where we use the odd 2-periodic extension of $f_s$ to define the dilation. For a sequence $\{s_n\}_{n=1}^\infty\in(0,c)^\infty$, let $g_{n}(x)=\frac{1}{n^\alpha} f_{s_n}(nx)$. Then, \begin{align*} g_n(x)&=\sum_{j=1}^\infty \frac{\langle f_{s_n},h_j\rangle_{\alpha}}{n^{\alpha}} h_j(nx)  = \sum_{j=1}^\infty  \frac{j^{2\alpha}\widehat{f_{s_n}}(j)\overline{\widehat{h_j}(j)}}{n^{\alpha}}h_{j}(nx) \\ & =\sum_{j=1}^\infty \frac{j^{\alpha} \widehat{f_{s_n}}(j)}{n^\alpha}\, \frac{\sqrt{2}\sin(jn\pi x)}{j^\alpha} = \sum_{j=1}^\infty j^\alpha \widehat{f_{s_n}}(j) h_{nj}(x)  = \sum_{j=1}^\infty M_j C_j h_n(x), \end{align*} where $M_jh_n=h_{jn}$ are spread shifts in $\mathrm{H}^\alpha$ and $C_j:\mathrm{H}^\alpha\longrightarrow \mathrm{H}^\alpha$ are the diagonal operators \[C_j=j^\alpha\operatorname{diag}\left[ \widehat{f_{s_n}}(j)\right]_{n=1}^\infty.\] Note that $C_1=I$ and $M_1=I$. Hence, $\{g_n\}_{n=1}^\infty\subset \mathrm{H}^\alpha$ is a Riesz basis if and only if the linear operator \begin{equation}\label{opT} T=I+\sum_{j=2}^\infty M_jC_j \end{equation} is invertible on $\mathrm{H}^\alpha$. 

In \cite{HLS}, Hedenmalm, Lindqvist and Seip, established necessary and sufficient conditions for the invertibility of $T$ for $\alpha=0$ in the case of a constant sequence, $s_n=s$ for all $n\in\mathbb{N}$. Concretely, according to \cite[Theorem 5.2]{HLS}, the system $\{f_s(n\,\cdot)\}_{n=1}^{\infty}$ is a Riesz basis of $\mathrm{H}^{0}$ if and only if the meromorphic extension of the function defined by the Dirichlet series, \[\mathfrak{d}(z) = \sum_{j=1}^{\infty}\widehat{f_s}(j) j^{-z},\] is bounded and achieves values away from zero in the half-plane $\{\operatorname{Re}(z)\geq 0\}$. As we shall see next, the corollaries~\ref{cor2} and \ref{BetterthanCristensen} provide specific extensions of this result to non-constant $\{s_n\}_{n=1}^\infty$ and general $\alpha\geq 0$. This is simply a re-formulation of a perturbation result dating back to Paley and Wiener, obtained by Christensen in \cite{Ch1995}.

\section{Dilated Weierstrass functions\label{sec3}}

By the above method, we can examine the linear completeness of trajectories of rough functions in $\mathrm{H}^\alpha$. The results presented below are surprising and at the same time they are illustrative.  

Let $\lambda\in(0,1)$ and let $p\not=1$ be a positive integer. Let \[W_{\lambda}(x)=\sqrt{2}\sum_{j=0}^\infty \lambda^j \sin(p^j \pi x) \] be Weierstrass functions. Since \[\widehat{W_\lambda}(k)=\begin{cases} 0 & k\not=p^j \\ \lambda^j & k=p^j, \end{cases}\] then $W_\lambda \in\mathrm{H}^\alpha$ if and only if $\lambda\in\big(0,p^{-\alpha}\big)$. Thus, in the framework of the previous section, set $c=1/p^{\alpha}$,  $\{\lambda_n\}_{n=1}^\infty\in \big(0,p^{-\alpha}\big)^\infty$ and  \begin{equation} \label{wn} w_n(x)=\frac{1}{n^\alpha} W_{\lambda_n}(nx).\end{equation} We study conditions for the family $\{w_n\}_{n=1}^\infty$ to be a Riesz basis of $\mathrm{H}^{\alpha}$.

To begin with, note that there is a direct argument for constant $\lambda_n=\lambda$ where $0<\lambda<1$, that $\{w_n\}$ is a basis of $\mathrm{H}^0$. Indeed, the Dirichlet series associated to $w_n$ is, \[\mathfrak{d}(s)=\sum_{j=0}^\infty \frac{\lambda^j}{p^{js}}=\sum_{j=0}^\infty \left(\frac{\lambda}{p^s}\right)^j.\] If $s$ is such that $\lambda<|p^s|$, then $\mathfrak{d}(s)=\frac{p^s}{p^s-\lambda}.$ But by unique analytic continuation, we know that this identity is true for all $s\in \C$ where the denominator does not vanish. Now, the latter happens if and only if $e^{s \log p}=\lambda.$ So the poles of $\mathfrak{d}$ lie along the line $\{s\in\C\,:\,\operatorname{Re}(s)=\log_p \lambda\}$. Since $\lambda<1$, this line is in the left hand side plane. Also, $\mathfrak{d}$ has no zeros. Hence, automatically by \cite[Theorem 5.2]{HLS}, we know that $\{w_n\}$ is a Riesz basis of $\mathrm{H}^0$ for $\lambda_n=\lambda$ constant.

The following extension of this claim to non-constant $\{\lambda_n\}$ and $\alpha\not=0$  is one of our main contributions in the present work. Its validity in the set $\mathtt{S}_1$ is surprising, in light of the fact that $\mathrm{H}^\alpha$ is the sharp regularity of $W_{\frac{1}{p^\alpha}}$ and that $\inf \lambda_n$ is allowed to vanish. 

\begin{thm} \label{theorem4.1} Let $\alpha\geq 0$. The set $\{w_n\}$ formed by the dilations \eqref{wn} of the Weierstrass functions with moving indices $\{\lambda_n\}$, is a Riesz basis of $\mathrm{H}^{\alpha}$ for all $\{\lambda_n\}\in \mathtt{S}_0 \cup \mathtt{S}_1$, where \[ \begin{aligned} \mathtt{S}_0&=\left\{\{\lambda_n\}\in \left(0,\frac{1}{2p^{\alpha}}\right)^{\infty}:\,\sup \lambda_n<\frac{1}{2p^{\alpha}}\right\} \qquad \text{and} \\ \mathtt{S}_1&=\left\{\{\lambda_n\}\in \left(0,\frac{1}{p^{\alpha}}\right)^{\infty}:\, \sup \lambda_n<\frac{1}{p^{\alpha}} \text{ and } \lambda_n=\lambda_{pn}\,\forall n\in\mathbb{N}\right\}. \end{aligned}\] \end{thm}  
\begin{proof}
Let $\{\lambda_n\}\in \mathtt{S}_0 \cup \mathtt{S}_1$  and $\mu=\sup_{n\in\mathbb{N}}\lambda_n$. Since $\mu<\frac{1}{p^\alpha}$, then the diagonal operators \[C_j=j^{\alpha}\operatorname{diag}\left[\widehat{W_{\lambda_n}}(j)\right]_{n=1}^{\infty}=\operatorname{diag}[c_j(n)]_{n=1}^\infty, \qquad c_j(n)=\begin{cases} 0 & j\not=p^l \\ \left(\lambda_n p^{\alpha}\right)^{l} & j=p^l, \end{cases}    \] with respect to the orthonormal basis $\{h_{n}\}_{n=1}^\infty$ of $\mathrm{H}^\alpha$ in equation \eqref{onb} are bounded and \[\|C_j\|=\begin{cases} 0 & j\not=p^l \\ (\mu p^{\alpha})^l & j=p^l. \end{cases}\] We split the proof into two cases. 

Firstly, assume that $\{\lambda_n\}\in \mathtt{S}_0$. Since $\mu<\frac{1}{2p^\alpha}$, then \[ \sum_{l=1}^\infty (\mu p^{\alpha})^l=\frac{p^\alpha \mu}{1-p^\alpha \mu}<1, \text{ hence }\sum_{j=2}^\infty \|C_j\|= \sum_{l=1}^\infty \|C_{p^l}\|<1.\] Thus, from the expression of the operator $T$ in \eqref{opT} and the fact that $M_p$ are isometries, it follows that $T:\mathrm{H}^\alpha \longrightarrow \mathrm{H}^\alpha$ is invertible and the claimed conclusion about $\{w_n\}$ follows in this case. 

Assume that $\{\lambda_n\}\in \mathtt{S}_1$. We aim to invoke Corollary~\ref{BetterthanCristensen}. For this purpose, fix $d\in\mathbb{N}$, $A_k=C_{p^k}$ for $1\leq k \leq d$ and \[B_j=\begin{cases} 0 & j\leq p^d \\ C_j & j>p^d.\end{cases}\] Then, on the left hand side of \eqref{pertcond} we have \begin{equation} \label{lhsof7} \sum_{j=2}^\infty \sup_{n\in\mathbb{N}} |b_j(n)|\leq \sum_{l=d+1}^{\infty} \left(\mu p^{\alpha}\right)^{l}=\sum_{l=d+1}^{\infty} \nu^l=\frac{\nu^{d+1}}{1-\nu},   \end{equation} where here and for the rest of the proof we set $\nu=\mu p^{\alpha}$. Now, we turn to the right hand side of \eqref{pertcond}. We have \begin{align*} \inf_{\substack{z\in\mathbb{D}\\ n\in\mathbb{N}}} |\alpha_n(z)|&= \inf_{n\in\mathbb{N}} \min_{\theta\in[0,2\pi]}\left|1+(\lambda_n p^\alpha) e^{i\theta}+\cdots+\left(\lambda_np^\alpha\right)^d e^{id\theta}\right| \\ &=\inf_{n\in\mathbb{N}} \min_{z\in \left(\lambda_n p^\alpha \mathbb{T}\right)} \left|1+z+\cdots+z^d\right|. \end{align*} Here $\mathbb{T}=\partial \mathbb{D}$. For any $z\in \left(\lambda_n p^\alpha \mathbb{T}\right)$ we have \[\left|1+z+\cdots +z^d\right|=\frac{\left|1-z^{d+1}\right|}{|1-z|}\geq \frac{1-\left(\lambda_np^\alpha\right)^{d+1}}{1+\lambda_np^\alpha}\geq \frac{1-\nu^{d+1}}{1+\nu}.\] Hence, \begin{equation} \label{rhsof7}\inf_{\substack{z\in\mathbb{D}\\ n\in\mathbb{N}}} |\alpha_n(z)|\geq \frac{1-\nu^{d+1}}{1+\nu}. \end{equation} The restriction $\mu<\frac{1}{p^\alpha}$ defining the set $\mathtt{S}_1$, implies that $\nu<1$. The right hand side of \eqref{lhsof7} has limit $0$, while the right hand side of \eqref{rhsof7} has limit $1$ as $d\to\infty$. Therefore, \eqref{pertcond} is valid for $d$ sufficiently large and so Corollary~\ref{BetterthanCristensen} gives the claimed conclusion in this case.\end{proof}

\begin{rmk} \label{relax} There are several ways to relax the conditions in the previous theorem. The conclusion is still valid, for example, if $\lambda_n=\lambda_{pn}$ for $n\geq N$ large enough. To show that this is indeed the case, note that by virtue of \cite[Theorem~2.20, p.265]{Kato}, we can change finitely many terms in the sequence $\{\lambda_n\}$ without altering the Riesz basis property of the corresponding family $\{w_n\}$, because this family is always \textit{w}-linearly independent.  \end{rmk}

By virtue of this remark, it is logical to conjecture that, most likely, $\{w_n\}$ is a Riesz basis for all $\{\lambda_n\}\in\big(0,\frac{1}{p^\alpha}\big)^\infty$, satisfying $\sup \lambda_n <\frac{1}{p^\alpha}$.


\section{Families of Gross-Pitaevskii eigenfunctions\label{sec4}}

We now turn to an analysis of the eigenfunctions of the non-linear Schr\"{o}dinger eigenvalue problem, \begin{equation} \label{Gross-Pitaevskii} \begin{aligned} &u''-u^3+\eta u=0 \\ &u(0)=u(1)=0,\end{aligned} \end{equation} for the stationary states of a Bose-Einstein condensate in a 1D square well trapping potential.   

The full family of eigenfunctions and eigenvalues of \eqref{Gross-Pitaevskii} is given in terms of Jacobi elliptic functions by \[u(x)=u_n(x,\mu)=2^{\frac32}n\mu K(\mu)\operatorname{sn}\big(2K(\mu)nx,\mu\big)\] and $\eta=\eta_n=4n^2(1+\mu^2)K(\mu)^2.$ Here $n\in\mathbb{N}$, $\mu\in(0,1)$ is the modulus, the complete integral \[K(\mu)=\int_0^1 (1-t^2)^{-\frac12}(1-\mu^2 t^2)^{-\frac12}\,\mathrm{d}t\] is the $\frac14$ period and  $\operatorname{sn}(y,\mu)$ is the odd $4K(\mu)$-periodic extension of the inverse of \[w(z)=\int_0^z (1-t^2)^{-\frac12}(1-\mu^2 t^2)^{-\frac12}\,\mathrm{d}t.\] The Fourier series of $\operatorname{sn}(\cdot,\mu)$ has a simple expression in terms of the nome $q\in(0,1)$, \[q=\exp\left[-\frac{K\big(\sqrt{1-\mu^2}\big)}{K(\mu)}\right].\]  With this expression, it follows that \begin{equation} \label{qn} u_n(x,\mu)=\frac{2^{\frac52}\pi nq^{\frac12}}{1-q}g(nx,q) \quad \text{for} \quad g(x,q)=\sum_{j=0}^\infty \frac{(1-q)q^j}{1-q^{2j+1}}\, e_{2j+1}(x).\end{equation} Note that, since  $\operatorname{sn}(\cdot,\mu)$ is $\operatorname{C}^\infty$ on the real axis, then $g(\cdot,q)\in\mathrm{H}^\alpha$ for all $\alpha\geq 0$. 

Take the trajectories in the family $\mathrm{G}=\{\Gamma_n\}$, parametrised in the variable $q\in(0,1)$ by the $n$-th dilation of $g$. Namely, take $\Gamma_n:q\longmapsto g(n \,\cdot ,q)$. Let us analyse the question of completeness in $\mathrm{H}^\alpha$ for $\mathrm{G}$. Concretely, let us examine conditions on sequences $\{q_n\}_{n=1}^\infty \in (0,1)^\infty$ for $\{g_n\}_{n=1}^\infty\subset \operatorname{H}^\alpha$ to be Riesz bases, where we write \begin{equation}\label{gn}g_n(x)=g(nx,q_n).\end{equation} If they are Riesz bases, then the corresponding families of eigenfunctions $\{u_n\}$ of \eqref{Gross-Pitaevskii} for corresponding moduli $\mu=\mu_n$ are also bases, and hence complete in $\operatorname{H}^\alpha$.

According to the results of \cite{boulton2023completeness}, we know that $\{g_n\}$ is a Riesz basis of  $\mathrm{H}^0$, provided that $\sup q_n<r_0(0)\approx 0.76806$. This threshold $r_0(0)$ is characterised analytically, as follows. By virtue of \cite[Theorem~3]{boulton2023completeness}, the operator \eqref{opT} associated to $\{g_n\}$ (case $\alpha=0$) is invertible, since \[\sum_{j=2}^\infty \|C_j\|\leq s_0\big(\sup q_n\big)-1<s_0(r_0(0))-1=1 , \] where the increasing function $s_0:(0,1)\longrightarrow \mathbb{R}$, is \begin{equation} \label{Lambert1} s_0(q)=\sum_{l=0}^\infty \frac{(1-q)q^l}{1-q^{2l+1}}=\frac{1-q}{\sqrt{q}}\left(L_1(\sqrt{q})-2L_1(q)+L_1(q^2)\right).\end{equation} Here \[L_1(r)=\sum_{n=1}^\infty \frac{r^n}{1-r^n}\] is a Lambert series, \cite[Section~3.7]{Borwein1999}. 

Our next goal now is two-folded. On the one hand, we examine the case $\alpha>0$. On the other hand, we show that the threshold $r_0(\alpha)$ that we obtain for the general case $\alpha\geq0$ is substantially improved, subject to conditions on $\{q_n\}_{n=1}^\infty$, via Corollary~\ref{BetterthanCristensen}. 

We first consider the question of completeness in Sobolev spaces. Since, \[ \widehat{g(\cdot,q)}(n)=\begin{cases} 0 & n=2 l \\ \frac{(1-q)q^l}{1-q^{2l+1}} & n=2l+1,\end{cases}\] then the associated operator \eqref{opT} in this case has coefficients \begin{equation} \label{matCj} C_j= \operatorname{diag}[c_j(n)]_{n=1}^\infty \quad\text{for} \quad c_j(n)=\begin{cases} 0 & j=2l \\ (2l+1)^{\alpha}\frac{(1-q_n)q_n^l}{1-q_n^{2l+1}} & j=2l+1. \end{cases}  \end{equation} Let \[s_\alpha(q)=\sum_{l=0}^\infty (2l+1)^{\alpha} \frac{(1-q)q^l}{1-q^{2l+1}}.\]  Since $q\longmapsto \frac{(1-q)q^l}{1-r^{2l+1}}$ is increasing, then $s_\alpha$ is an increasing function in $q$ and $\alpha$. Thus, the following is valid.

\begin{thm} \label{oneterm}  Let $\alpha\geq 0$. Let $r_0(\alpha)\in (0,1)$ be the unique solutions to the equation $s_{\alpha}(q)=2$. Let \[\mathtt{T}_0=\big\{\{q_n\}\in(0,1)^\infty\,:\,\sup q_n<r_0(\alpha)\big\}.\] Then, the set of eigenfunctions $\big\{u_n(\cdot,\mu_n)\big\}$ of the Gross-Pitaevskii equation \eqref{Gross-Pitaevskii} is a basis of $\mathrm{H}^\alpha$ for all sequences of moduli $\{\mu_n\}$ associated to nomes $\{q_n\}\in \mathtt{T}_0$.   \end{thm}
\begin{proof} Recall the expressions \eqref{qn} and \eqref{gn}. Observe that, similar to the case $\alpha=0$, we have \[\sum_{j=1}^\infty \|C_j\|\leq s_{\alpha}(\sup q_n)-1<1.\] Therefore, $\{g_n\}\subset \operatorname{H}^\alpha$ is a Riesz basis. Thus, $\{u_n(\cdot,\mu_n)\}\subset \operatorname{H}^\alpha$ is indeed a basis. \end{proof}

We now characterise the threshold $r_0(\alpha)$ of the previous theorem, in terms of Lambert series.  For any $f:[0,\infty)\longrightarrow [0,\infty)$ and $r\in(0,1)$, denote by  \[L_f(r)=\sum_{n=1}^\infty f(n) \frac{r^n}{1-r^n}\] the associated (generalised) Lambert series, \cite[p.99]{Borwein1999}. Splitting the left hand side into even and odd terms, gives \[ \sum_{n=1}^\infty f(n) \frac{r^n}{1-r^{2n}}=L_f(r)-L_{f}(r^2).\] Let $f_2(n)=f(2n)$. Then, by applying the previous identity twice, we have that \begin{align*} L_f(r)-L_{f}(r^2)&-L_{f_2}(r^2)+L_{f_2}(r^4) =\sum_{j=0}^\infty f(2j+1)\frac{r^{2j+1}}{1-r^{4j+2}}.\end{align*} Hence, we have the representation \[ s_{\alpha}(q)=\frac{1-q}{\sqrt{q}}\left(L_f(\sqrt{q})-L_{f}(q)-L_{f_2}(q)+L_{f_2}(q^2)\right)\] for $f(n)=n^\alpha$ and $f_2(n)=(2n)^\alpha$. Note that this matches \eqref{Lambert1} for $f(n)=1$ and $\alpha=0$.

\begin{rmk} \label{alpha1} Evaluating the Fourier series  \cite[22.11.13]{NIST} of $\operatorname{sn}(x,\mu)^2$ at $x=0$, leads to the identity \[\sum_{n=1}^\infty n \frac{r^n}{1-r^{2n}}=\frac{K(\nu)}{2\pi^2}\big(K(\nu)-E(\nu)\big),\] where $\nu$ is the nome associated to modulus $r$ and $E(\nu)$ is the complete elliptic integral of second kind \cite[19.2.8]{NIST}. Then, since \[s_1(q)=\frac{1-q}{\sqrt{q}}\left(\sum_{n=1}^\infty n \frac{q^{\frac{n}{2}}}{1-q^{n}} - 2\sum_{n=1}^\infty n \frac{q^{n}}{1-q^{2n}}\right),\] we obtain a close expression for $s_1$, and hence $r_0(1)$, in terms of the two classical elliptic integrals. Combining this with monotonicity in $\alpha$, might lead to accessible analytic upper bounds for $r_0(1)$. \end{rmk}

\medskip

We now employ Corollary~\ref{BetterthanCristensen} in order to show that Theorem~\ref{oneterm} is not optimal. The next statement is our main result in this respect.
In Figure~\ref{fig2} (left) we present a graph comparing numerical approximations of the parameters involved in the statement. It strongly indicates that, in the segment $\alpha\in[0,2]$, the value of $r_1(\alpha)$ is a noticeable improvement  with respect to the threshold $r_0(\alpha)$ of Theorem~\ref{oneterm}  near $\alpha=0$.    

\begin{thm} \label{twoterm} Let $p\not=1$ be a positive integer. For $\alpha\geq 0$, let $r_1(\alpha)>0$ be the unique solution to the equation \[ s_\alpha(q)=2+\frac{p^\alpha(1-q)q^{\frac{p^2-1}{2}}}{2(1-q^{p^2})} .\] Let \[\mathtt{T}_1=\left\{\{q_n\}\in\big(0,1\big)^\infty\,:\, \sup q_n<r_1(\alpha)\text{ and } q_n=q_{pn}\,\forall n\in\mathbb{N} \right\}.\]  Then, the set of eigenfunctions $\big\{u_n(\cdot,\mu_n)\big\}$ of the Gross-Pitaevskii equation \eqref{Gross-Pitaevskii} is a basis of $\mathrm{H}^\alpha$ for all sequences of moduli $\{\mu_n\}$ associated to nomes $\{q_n\}\in \mathtt{T}_1$. \end{thm}

\begin{figure} \includegraphics[width=0.46\textwidth]{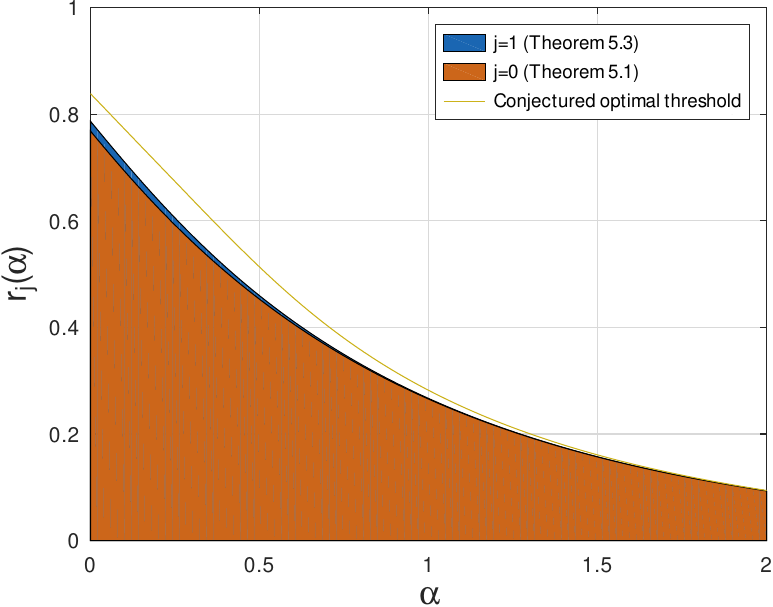}\quad  \includegraphics[width=0.46\textwidth]{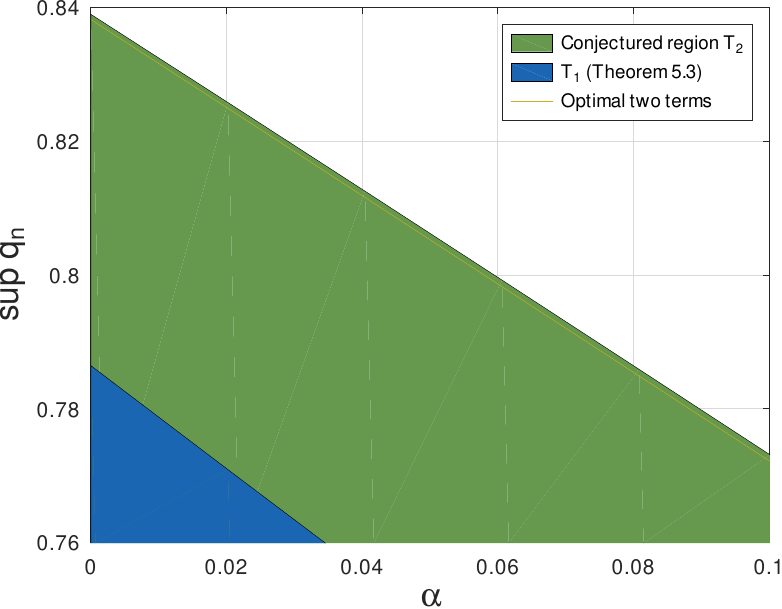} \caption{Illustration of the regions $\mathtt{T}_j$ and thresholds $r_j(\alpha)$ covered by theorems~\ref{oneterm} and \ref{twoterm}, and Remark~\ref{remarkTj} for $p=3$: (left) $\alpha\in [0,2]$ and $j=0,1$ alongside with the conjectured optimal threshold $\tilde{r}_1(\alpha)$ and (right) $\alpha\in [0,0.1]$ and $j=1,2$. We obtained the numerical approximation of $s_{\alpha}(q)$ taking $500$ terms in the summation. \label{fig2}} \end{figure} 

\begin{proof}  We drop the index $\alpha$, writing $r_1$ instead of $r_1(\alpha)$. The strategy is to take in Corollary~\ref{BetterthanCristensen};  $d=2$, $A_1=C_p$, $A_2=C_{p^2}$ and $B_j=C_j$ for $j\not=p,p^2$, where the diagonal matrices $C_j$ are given by \eqref{matCj}. We show that this strategy applies.  

Let \[a(q)=p^\alpha\frac{(1-q)q^{\frac{p-1}{2}}}{1-q^p} \quad\text{and} \quad b(q)=p^{2\alpha}\frac{(1-q)q^{\frac{p^2-1}{2}}}{1-q^{p^2}}.\]Then, both these quantities are increasing as functions of $q\in(0,1)$. Also, \[ \frac{b(q)}{p^{\alpha}}<a(q)   \qquad \text{ and  } \qquad s_\alpha(r_1)=2+\frac{b(r_1)}{2  p^\alpha}\]  

Let $\{q_n\}\in \mathtt{T}_1$. We first show that $\big(a(q_n),b(q_n)\big)\in\mathbb{G}_2$ for all $n\in\mathbb{N}$. Indeed, since\[s_\alpha(r_1)=\sum_{l\in \mathbb{N}\setminus\{\frac{p-1}{2},\frac{p^2-1}{2}\}} (2l+1)^{\alpha} \frac{(1-r_1)r_1^l}{1-r_1^{2l+1}}+1+a(r_1)+b(r_1),\] then $a(r_1)+b(r_1)\left(1-\frac{1}{2  p^\alpha}\right)<1.$ Thus, $b(r_1)<1$ and \begin{equation} \label{inG}\begin{aligned}1-a(q_n)+b(q_n)&> 1-a(q_n)-b(q_n)\left(1-\frac{1}{2  p^{\alpha}}\right)\\ &\geq 1-a(r_1)-b(r_1)\left(1-\frac{1}{2  p^{\alpha}}\right)>0 .\end{aligned} \end{equation} Hence, $\big(a(q_n),b(q_n)\big)\in\mathbb{G}_2$. See Figure~\ref{fig1} (left). 

 We now probe the hypothesis of Corollary~\ref{BetterthanCristensen}. By positivity and monotonicity of the coefficients in the series, we have that the left hand side of \eqref{pertcond} has the following upper bound,\begin{align*} \sum_{j=2}^\infty \sup_{n\in\mathbb{N}} |b_j(n)| & = \sum_{l\in \mathbb{N}\setminus\{\frac{p-1}{2},\frac{p^2-1}{2}\}}(2l+1)^{\alpha}  \sup_{n\in \mathbb{N}} \frac{(1-q_n)q_n^l}{1-q_n^{2l+1}}\\   &\leq \sum_{l\in \mathbb{N}\setminus\{\frac{p-1}{2},\frac{p^2-1}{2}\}} (2l+1)^{\alpha} \frac{(1-r_1)r_1^l}{1-r_1^{2l+1}}  \\ &= s_\alpha(r_1) -1-a(r_1)-b(r_1) \\ & =1-a(r_1)-b(r_1)\left(1-\frac{1}{2  p^{\alpha}}\right).\end{align*}

Now, for $(a,b)\in\mathbb{G}_2\cap [0,\infty)^2$, the maximum principle gives \begin{align}\min_{z\in\mathbb{D}}|1+az+bz^2|& =\begin{cases} 1-a+b & \frac{a(b+1)}{4b}\geq  1  \\ (1-b)\sqrt{1-\frac{a^2}{4b}} & \text{otherwise}.\end{cases}\label{casesmin}\end{align} Then, the right hand side of \eqref{pertcond}, obeys the following branching for the different $n\in\mathbb{N}$. In the case $\frac{a(q_n)(b(q_n)+1)}{4b(q_n)}\geq 1$, we have \[\min_{z\in\mathbb{D}}\big|1+a(q_n)z+b(q_n)z^2\big|=1-a(q_n)+b(q_n)> 1-a(r_1)-b(r_1)\left(1-\frac{1}{2  p^{\alpha}}\right) \] as we saw from \eqref{inG}, therefore \eqref{pertcond} holds. Otherwise, since $ \sqrt{1-\frac{a^2}{4b}}>1-\frac{a}{2} $ whenever $b>\frac{a}{4-a}$, we have \begin{align*} \min_{z\in\mathbb{D}}&\big|1+a(q_n)z+b(q_n)z^2\big|=\big(1-b(q_n)\big)\sqrt{1-\frac{a(q_n)^2}{4b(q_n)}} \\& > \big(1-b(q_n)\big)\left(1-\frac{a(q_n)}{2}\right) \geq \big(1-b(r_1)\big)\left(1-\frac{a(r_1)}{2}\right) \\&>  1-a(r_1)-b(r_1)\left(1-\frac{1}{2  p^{\alpha}}\right) , \end{align*} so that once again \eqref{pertcond} holds. This completes the proof of the claim.  \end{proof}

An observation similar to the one given in Remark~\ref{relax} applies in the context of this theorem. That is, the conclusion is still valid, under the less restrictive condition that $\{q_n\}$ is $p$-periodic for $n$ sufficiently large and \[\limsup_{n\to \infty}q_n<r_1(\alpha).\] Moreover, if $q_n=q_{rn}$ for $r$ a composite number, then the theorem applies with $p$ any prime factors of $r$. In that case, a larger region $\mathtt{T}_1$ will be achieved, if we take the smallest prime $p$ among these factor. 

We expect that the claim is still valid for $\sup q_n<\tilde{r}_1(\alpha)$ where the latter is the solution to \[s_\alpha(q)=1+a(q)+b(q)+\min_{z\in \mathbb{D}}\left|1+a(q)z+b(q)z^2\right|.\] However, the verification of this requires knowing sharp bounds for $s_\alpha(r_0(\alpha))$. See Remark~\ref{alpha1} for a development in the case $\alpha=1$. Moreover, this might lead to significant improvements upon the statement above only for $\alpha$ close to 0. The graph shown in Figure~\ref{fig2} (left) confirms this with numerical approximations for $p=3$.

\begin{rmk} \label{remarkTj} By increasing $d$ in Corollary~\ref{BetterthanCristensen}, and developing the argument in the proof of Theorem~\ref{twoterm}, one might obtain an increasing family of nested regions $\mathtt{T}_{d-1}$, where $\{g_n\}$ are Riesz bases, under the periodicity condition $q_n=q_{pn}$. Compare the two graphs shown in Figure~\ref{fig1} for a visualisation of $\mathbb{G}_2$ vs $\mathbb{G}_3$, then see Figure~\ref{fig2} (right). Our numerical simulations suggest that, taking $d=3$, for example, leads to improvements over the region $\mathtt{T}_1$ in a small neighbourhood of $\alpha=0$.     \end{rmk} 


\appendix

\section{Model characterisation of $\mathbb{G}_d$ \label{appa}}
In this appendix we discuss the relation between the hypotheses of the corollaries~\ref{cor2} and \ref{BetterthanCristensen}, with the theory surrounding models and generalised Schur-Cohn theorems. Some of the material presented below is standard, but we believe that Lemma~\ref{polydisctheorem} and Corollary~\ref{mainapp} are new. 

We call a matrix $Y\in\mathbb{C}^{d\times d}$, such that $\|Y\|< 1$, $\operatorname{Spec} Y\subset \mathbb{D}$ and $\operatorname{rank}\big(1-Y^{*}Y\big)=1$, a \emph{Pt\'{a}k-Young matrix}. Up to unitary equivalence, all such matrices are characterised by the upper triangular representation \begin{equation}\label{uppertriangularT} Y=\begin{bmatrix} \beta_{1} & s_{1}s_{2} & -s_{1}\overline{\beta_{2}}s_{2} & \hdots & (-1)^{d}s_{1}\overline{\beta_{2}}\hdots \overline{\beta_{d-1}}s_{d} \\ 0 & \beta_{2} & s_{2}s_{3} & \hdots & (-1)^{d-1}s_{2}\overline{\beta_{3}}\hdots \overline{\beta_{d-1}}s_{d}\\ \vdots & \ddots & \ddots & \ddots  & \vdots \\ \vdots & & \ddots & \beta_{d-1} & s_{d-1}s_d \\ 0 & \hdots & \hdots & 0 & \beta_{d} \\ \end{bmatrix} \end{equation} where $\beta_{j}\in\mathbb{D}$ and $s_{j}=(1-\lvert \beta_{j} \rvert^{2})^{1/2}$ for $j=1,\hdots,d$. See \cite[Remarks 2 and 3]{ptakyoung}. 

Let the scalar polynomial \begin{equation} \label{polp} P(z)=z^{d}+c_{1}z^{d-1}+\hdots+c_{d}\end{equation} with conjugate $Q(z)=z^{d}\overline{P(1/\overline{z})}$ and coefficients vector $\underline{c}=(c_1,\ldots,c_d)$, as in Section~\ref{sec1}. Let $\mathbf{H}:\mathbb{C}^d\longrightarrow \mathbb{R}$ be the Hermitian form given by \begin{equation} \label{rawH}\mathbf{H}(\underline{x}) = \|Q(Y)\underline{x}\|^{2}-\|P(Y)\underline{x}\|^{2},\end{equation} for any $Y\in\mathbb{C}^{d\times d}$ a Pt\'{a}k-Young matrix.  By virtue of the generalised Schur-Cohn theorem obtained in \cite{ptakyoung}, it follows that $\underline{c}\in\mathbb{G}_d$ if and only if $\mathbf{H}$ is positive definite for one (and hence all) such matrix $Y$. If this is the case, then the finite Blaschke product, \begin{equation}\label{blaschkeeq}B(z)=\frac{P(z)}{Q(z)}=\prod_{j=1}^{d}\dfrac{z+\lambda_{j}}{1+\overline{\lambda_{j}}z},\end{equation} is uniformly bounded by 1 on $\operatorname{cl}(\mathbb{D})$. 

Let \[\mathcal{S}(\mathbb{D})=\left\{f:\mathbb{D}\longrightarrow \mathbb{C}\text{ holomorphic }:\,  \sup_{z\in\mathbb{D}}|f(z)|\leq 1 \right\}.\] be the \emph{Schur class of the unit disk}. We say that a {\it{model}} for $f\in\mathcal{S}(\mathbb{D})$, is a pair $(\mathcal{M},u)$ where $\mathcal{M}$ is a separable Hilbert space and $u:\mathbb{D}\longrightarrow \mathcal{M}$ is a holomorphic family, such that \begin{equation}\label{modeleq} 1-\overline{f(w)}f(z) = \Big\langle (1-\overline{w}z)u(z),u(w)\Big\rangle_{\mathcal{M}} \end{equation} for all $w,\,z\in\mathbb{D}$. See \cite{AMYOperatorbook}, also \cite{Agler} where this concept originated. According to the results of the recent work \cite{AELY}, there exists a model formula for functions defined on $\mathbb{D}^{d}$ and such models can be used to study boundary point singularities and directional derivatives of holomorphic functions on $\mathbb{D}^{d}$. See \cite{AglerTullyDoyleYoung} for results on $\mathbb{D}^{2}$. 

The Blaschke product \eqref{blaschkeeq} lies in $\mathcal{S}(\mathbb{D})$ and, according to the representation given in \cite[Lemma 2.47]{AMYOperatorbook}, it has a model. Indeed, for $j=1,\hdots,d$, let \begin{equation}\label{TM} E_{j}(z)=\dfrac{(1-\lvert\lambda_{j}\rvert^{2})^{\frac{1}{2}}}{1+\overline{\lambda_{j}}z}\prod_{k=1}^{j-1}\dfrac{z+\lambda_{k}}{1+\overline{\lambda_{k}}z}.  \end{equation} Then, $\{E_{j}\}_{j=1}^d\subset \mathcal{S}(\mathbb{D})$ is an orthonormal family in $\mathrm{H}^{2}(\mathbb{D};\mathbb{C})$ and \begin{equation}\label{blaschkeproductequation} \begin{aligned} 1-\overline{B(w)}B(z)&=\sum_{j=1}^{d}\overline{E_{j}(w)}(1-\overline{w}z)E_{j}(z) =\Big\langle(1-\overline{w}z)u(z),u(w)\Big\rangle_{\mathbb{C}^{d}} \end{aligned} \end{equation} for all $z,w\in\mathbb{D}$. Therefore, the pair $(\mathcal{M},u)$ with $\mathcal{M}=\mathbb{C}^{d}$ and $u(z) = (E_{1}(z),\hdots,E_{d}(z))$, is indeed a model for the Blaschke product \eqref{blaschkeeq}. 

The functions $E_{j}\in\mathcal{S}(\mathbb{D})$ are often referred-to as the \emph{Takenaka-Malmquist} functions associated to \eqref{blaschkeeq}, \cite{TM}. We now give the expression of the Hermitian form $\mathbf{H}$ in terms of these functions.  

\begin{lem}\label{polydisctheorem} Let the scalar polynomial $P$ be as in the equation \eqref{polp}, with its conjugate $Q$, be such that the coefficients $\underline{c}\in\mathbb{G}_d$. Let $\{E_j\}_{j=1}^d$ be the Takenaka-Malmquist functions associated to the Blaschke product $B=\frac{P}{Q}\in\mathcal{S}(\mathbb{D})$. Then, the Hermitian form in \eqref{rawH} for any fixed Pt\'{a}k-Young matrix $Y$, has the representation \begin{equation}\label{normyequalsinnerprod} \mathbf{H}(\underline{x})=\sum_{j=1}^{d}\big\langle Q^*(Y)E^{*}_{j}(Y)(1-Y^{*}Y)E_{j}(Y)Q(Y)\underline{x},\underline{x}\big\rangle .\end{equation} \end{lem}
\begin{proof} By applying the Riesz-Dunford functional calculus formula, we have that \[I-B^*(Y)B(Y)=\sum_{j=1}^\infty E_j^*(Y)(I-Y^*Y)E_j(Y).\] Then, observe that $Q^*(Y)[I-B^*(Y)B(Y)]Q(Y)=Q^*(Y)Q(Y)-P^*(Y)P(Y)$ and that $\mathbf{H}(\underline{x})=\left\langle \big[ Q^*(Y)Q(Y)-P^*(Y)P(Y)\big]\underline{x},\underline{x}\right\rangle.$  
 \end{proof}

We now give an expression for the operator associated to the Hermitian form \eqref{rawH} in terms of a realisation. As we shall see, we find this realisation by ``re-shuffling'' the model formula in \eqref{modeleq} and appealing to the classical ``lurking isometry'' argument. Cf. \cite[Section 2.4]{AMYOperatorbook}. The next statement is our main contribution in this appendix.

\begin{cor} \label{mainapp} Let $Y$ be a Pt\'{a}k-Young matrix on $\mathbb{C}^{d}$.   Let \[u(Y)=(E_{1}(Y),\ldots,E_{d}(Y))\in\mathcal{B}(\mathbb{C}^{d^{2}})\] where $\{E_{j}\}_{j=1}^{d}$ are the Takenaka-Malmquist functions associated to the finite Blaschke product $B=\frac{P}{Q}$ in \eqref{blaschkeeq}. Then, there exists $H(Y)\in\mathcal{B}(\mathbb{C}^{d^{2}},\mathbb{C}^{d})$ such that \[\mathbf{H}(\underline{x})=\|H(Y)u(Y)Q(Y)\underline{x}\|^2,\] for all $\underline{x}\in\mathbb{C}^{d}$. \end{cor}

 \begin{proof}Let $\underline{z}=Q(Y)\underline{x}$ and prescribe $\|\underline{y}\|^2=\mathbf{H}(\underline{x})$ for conditions on $\underline{y}$ to be found. By virtue of \eqref{normyequalsinnerprod}, \[\|\underline{y}\|^2=\sum_{j=1}^d\langle (I-Y^*Y)E_j(Y)\underline{z},E_j(Y)\underline{z}\rangle.\] Then, \begin{align*} \|\underline{y}\|^2&=\left\langle \begin{bmatrix} (I-Y^{*}Y)E_{1}(Y) \underline{z} \\ \vdots \\ (I-Y^{*}Y)E_{d}(Y) \underline{z} \end{bmatrix}, \begin{bmatrix} E_{1}(Y)\underline{z} \\ \vdots \\ E_{d}(Y)\underline{z} \end{bmatrix} \right\rangle_{\mathbb{C}^{d^{2}}} \\ &=\|u(Y)\underline{z}\|^2-\|(I\otimes Y)u(Y)\underline{z}\|^2,  \end{align*} where $u(Y)$ is as in the hypothesis. Indeed, note that \[(I\otimes Y)u(Y) \begin{bmatrix}\underline{z} \\ \vdots \\ \underline{z} \end{bmatrix} = \begin{bmatrix} Y &  \hdots & 0\\ \vdots & \ddots & \vdots\\ 0 & \hdots & Y \end{bmatrix}\begin{bmatrix} E_{1}(Y) \underline{z}\\ \vdots\\ E_{d}(Y) \underline{z} \end{bmatrix} = \begin{bmatrix} YE_{1}(Y)\underline{z}\\ \vdots\\ YE_{d}(Y)\underline{z} \end{bmatrix}.\] 

Write $\underline{0}=(0,\ldots,0)\in\mathbb{C}^{d}$. Then, \[\left\|\begin{bmatrix} \underline{y} \\ (I\otimes Y)u(Y)\underline{z} \end{bmatrix} \right\|_{\mathbb{C}^{d}\oplus\mathbb{C}^{d^{2}}} = \left\|\begin{bmatrix} \underline{0}\\ u(Y)\underline{z} \end{bmatrix}\right\|_{\mathbb{C}^{d}\oplus\mathbb{C}^{d^{2}}}.  \] From this, it follows that, the two families of vectors \[\begin{bmatrix} \underline{y} \\ (I\otimes Y)u(Y)\underline{z} \end{bmatrix}~~\text{and}~~\begin{bmatrix} \underline{0}\\ u(Y)\underline{z} \end{bmatrix}\] have the same Gramians in $\mathbb{C}^{d}\oplus\mathbb{C}^{d^{2}}$. Hence, by the lurking isometry argument, \cite[Section 2.4]{AMYOperatorbook}, there exists a linear isometry $L(Y)$ such that \[L(Y):\mathrm{Span}\bigg\{\begin{bmatrix} \underline{0}\\ u(Y)\underline{z} \end{bmatrix} : \underline{z}\in\mathbb{C}^{d}\bigg\}\longrightarrow \mathrm{Span}\bigg\{\begin{bmatrix} \underline{y} \\ (I \otimes Y)u(Y)\underline{z} \end{bmatrix} : \underline{z}\in\mathbb{C}^{d}\bigg\}\]  and \[L(Y)\begin{bmatrix} \underline{0}\\ u(Y)\underline{z} \end{bmatrix} = \begin{bmatrix} \underline{y} \\ (I\otimes Y)u(Y)\underline{z} \end{bmatrix} \] for all $\underline{z}\in\mathbb{C}^{d}$. The isometry $L(Y)$ extends to a unitary operator on $\mathbb{C}^d\oplus \mathbb{C}^{d^2}$, with block matrix form \[\begin{bmatrix} F(Y) & H(Y) \\ G(Y) & K(Y)\end{bmatrix}.\] Then, indeed $\underline{y} =H(Y)u(Y)\underline{z}$ for  suitable $H(Y)\in \mathcal{B}(\mathbb{C}^{d^2},\mathbb{C}^d)$ as required. \end{proof}

A comment about the other hypotheses of Corollary~\ref{BetterthanCristensen} is now in place. We saw that the positivity of the Hermitian form $\mathbf{H}$ characterises the polydisk $\mathbb{G}_d$. Thus, it might be natural to ask whether the infimum on the right hand side of the condition \eqref{pertcond} could also be described through $\mathbf{H}$. Unfortunately, simple examples with $d=2$ show that this is not the case.

\bibliographystyle{amsplain}
\bibliography{bib}

\vfill

\noindent Lyonell Boulton (\texttt{l.boulton@hw.ac.uk})

\noindent Department of Mathematics and Maxwell Institute for Mathematical Sciences, 
Heriot-Watt University, Edinburgh EH14 4AS, UK.

\

\noindent Connor Evans (\texttt{CEvansMathematics@outlook.com})

\noindent North Brent School, London NW10 2UF, UK. 

\noindent Part of this work was conducted while this author was a PhD student at the University of Newcastle upon Tyne. EPSRC grant \textbf{DTR~21~EP/T517914/1}.

\end{document}